\documentclass{amsart}
\usepackage{amsfonts,amssymb,amsmath,amsthm}
\usepackage{url}
\usepackage{enumerate}
\usepackage{bbm}
\usepackage{amssymb}
\usepackage{mathrsfs}
\usepackage[usenames]{xcolor}

\newtheorem{thrm}{Theorem}[section]
\newtheorem{lem}[thrm]{Lemma}
\newtheorem{prop}[thrm]{Proposition}
\newtheorem{cor}[thrm]{Corollary}

\theoremstyle{definition}

\newtheorem{remark}[thrm]{Remark}
\numberwithin{equation}{section}

\DeclareMathOperator{\conv}{conv}

\DeclareMathOperator{\id}{id}

\newcommand{\N}{\ensuremath{{\mathbb N}}}

\newcommand{\R}{\ensuremath{{\mathbb R}}}
\newcommand{\E}{\ensuremath{{\mathbb E}}}
\newcommand{\Pro}{\ensuremath{{\mathbb P}}}

\newcommand{\abs}[1]{\left\lvert#1 \right\rvert}

\newcommand{\vol}{\mathrm{vol}}
\newcommand\rad{{\rm rad}}

\newcommand\Lip{{\rm Lip}}

\newcommand{\BB}{\mathbb{B}}
\newcommand{\DD}{\mathbb{D}}
\newcommand{\SSS}{\mathbb{S}}
\newcommand{\calP}{\mathcal{P}}
\newcommand{\calC}{\mathcal{C}}
\newcommand{\eps}{\varepsilon}

\newcommand{\bM}{\ensuremath{{\mathbf m}}}
\newcommand{\dint}{\,{\rm d}}

\author{ 
  Aicke Hinrichs
  \and 
  Joscha Prochno 
  \and 
  Mario Ullrich
}
\address[Aicke Hinrichs]{Institut f\"ur Analysis\\
Johannes Kepler Universit\"at Linz\\
Altenbergerstrasse 69\\
4040 Linz\\
Austria}
\email{aicke.hinrichs@jku.at}
\address[Mario Ullrich]{Institut f\"ur Analysis\\
Johannes Kepler Universit\"at Linz\\
Altenbergerstrasse 69\\
4040 Linz\\
Austria}
\email{mario.ullrich@jku.at}
\address[Joscha Prochno]{School of Mathematics \& Physical Sciences\\
University of Hull\\
Cottingham Road\\
Hull HU6 7RX \\
United Kingdom}
\email{j.prochno@hull.ac.uk}
\thanks{The authors gratefully acknowledge the support of the Erwin Schr\"odinger International Institute for Mathematics and Physics (ESI) in Vienna under the thematic programme ``Tractability of High Dimensional Problems and Discrepancy'', where part of this work was completed. A. Hinrichs is supported by the Austrian Science Fund (FWF) Project F5509-N26, which is a part of the Special Research Program ``Quasi-Monte Carlo Methods: Theory and Applications''. J. Prochno is supported by a Visiting Professorship of the Ruhr University Bochum.} 


\begin{document}

\title[Curse of dimensionality for numerical integration]{The curse of dimensionality for numerical integration on general domains}

\begin{abstract}
We prove the curse of dimensionality in the worst case setting for multivariate 
numerical integration for various classes of smooth functions. 
We prove the results when the domains are isotropic convex bodies with small 
diameter satisfying a universal $\psi_2$-estimate. 
In particular, we obtain the result for the important class of volume-normalized $\ell_p^d$-balls in the complete regime $2\leq p \leq \infty$. 
This extends a result in a work of A. Hinrichs, E. Novak, M. Ullrich and H. Wo\'zniakowski \cite{HNUW2} to the whole range $2\leq p \leq \infty$,
and additionally provides a unified approach. 
The key ingredient in the proof is a deep result from the theory of Asymptotic 
Geometric Analysis, the thin-shell volume concentration estimate due to 
O. Gu\'edon and E. Milman. The connection of Asymptotic Geometric Analysis and Information-based Complexity revealed in this work seems promising and is of independent interest. 
\end{abstract}
\maketitle


\section{Introduction and Main results} \label{intro and main}

\subsection{Introduction} 
In the last decade, understanding complex systems that depend on a huge amount of 
parameters and, more specifically, the study of high-dimensional geometric structures has become increasingly important. It has become apparent by now that the presence of high dimensions forces a certain regularity in the geometry of the space while, on the other hand, it unfolds various rather unexpected phenomena. Two independent and rather young mathematical disciplines that center around questions of this type from different perspectives are Information-based Complexity and Asymptotic Geometric Analysis.  This work will bring together both areas for the first time to tackle one of the most frequent questions arising in and related to high-dimensional frameworks, namely the one for the curse of dimensionality. 

A prominent example in this respect is multivariate numerical integration, which has received a lot of attention in previous years and is the central topic of this work. One is interested in approximating, up to some given error $\varepsilon>0$, the value of a multivariate integral over a volume-normalized domain $K_d\subseteq \R^d$ ($d$ considered large) of a function $f:K_d\to\R$ that belongs to some class $\mathscr F_d$ of smooth functions. In this manuscript, we shall consider linear (deterministic) algorithms that use only function values to approximate such an integral. If the minimal number of function evaluations needed for this task in the 
worst case setting increases exponentially with the space dimension $d$, 
then we say that the function class $\mathscr F_d$ suffers from the curse of 
dimensionality. For any unexplained notion or notation, we refer the reader to Section \ref{Sec:prelim and notation}.

In \cite{HNUW2}, improving and extending previous results obtained in 
\cite{HNUW1}, Hinrichs, Novak, Ullrich and Wo\'zniakowski proved the curse of dimensionality in the worst case 
setting for numerical integration for a number of classes of smooth $d$-variate 
functions. They considered different bounds on the Lipschitz constants for the 
directional or partial derivatives of $f\in \mathscr C^k(K_d)$ 
($k\in\N\cup\{+\infty\}$), where $\mathscr C^k(K_d)$ denotes the space of $k$-times continuously differentiable functions on a domain $K_d\subseteq \R^d$ with $\vol_d(K_d)=1$. The assumption that the domain of integration has volume one guarantees that the integration problem is properly normalized. To be more specific, for $d\in\N$ and a volume-normalized domain $K_d\subseteq\R^d$, we define the class $\mathscr C^1_d(A_d,B_d,K_d)$ of smooth functions depending on two Lipschitz parameters, $A_d>0$ and $B_d>0$, as follows:
\begin{align*}
& \mathscr C^1_d(A_d,B_d,K_d) \,:=\, \cr
 & \qquad \Big\{f\in \mathscr C^1(K_d)\,\colon\,
		\|f\|_\infty\le1,\,\Lip(f)\le A_{d} \text{ and }  \Lip(D^\theta f)\le B_{d} \text{ for all } \theta\in\SSS^{d-1} \Big\},
\end{align*}
where $D^\theta$ denotes the directional derivative in direction $\theta\in\SSS^{d-1}$ 
and
\[
\Lip(g) \,:=\, \sup_{x,y\in K_d}\frac{|g(x)-g(y)|}{\|x-y\|_2}.
\]
Classes of this type and variants thereof are classical objects in 
numerical analysis.
Historically, the central question was the rate of error convergence for fixed dimension $d$. 
Only recently, motivated by modern applications, the emphasis was shifted to 
the study of the dependence of the error on the dimension. 
Specifically, in \cite{HNUW1, HNUW2} the authors discussed necessary and 
sufficient conditions 
on the parameters $A_d$ and $B_d$ as well as on the domain $K_d$ such that the curse of dimensionality holds for the class $\mathscr C^1_d(A_d,B_d,K_d)$. 
The results from their paper \cite{HNUW2} are sharp and 
characterize the curse of dimensionality if the domain of integration is 
either the cube, i.e., $K_d=[0,1]^d$, or a convex body with the property
\begin{align}\label{eq:small radius}
\limsup_{d\to\infty} \frac{\rad(K_d)}{\sqrt{d}} < \sqrt{\frac{2}{\pi e}}\,,
\end{align}
where $\rad(K_d)$ denotes the radius of the set $K_d$ (see~\cite[Theorem~4.1]{HNUW2}). 
The latter are sets with rather small radius and include, for instance, 
the case of $\ell_p$-ball domains if $2\leq p < p_0$ where $p_0\approx 170.5186$. 
However, this not only leaves a gap for the remaining class of $\ell_p$-ball 
domains, but also the methods of proof are different for cubes and sets 
satisfying~\eqref{eq:small radius}. 
The main purpose of the present paper is to present a novel and unified 
approach to this problem that will, at the same time, allow us to 
complement and extend the results from \cite{HNUW1,HNUW2} to a wider class 
of domains of integration. In particular, we study the case where these bodies 
satisfy a \emph{uniform $\psi_\alpha$-estimate}, $\alpha\in[1,2]$ (see Section~\ref{sec:BackgroundB} and Proposition~\ref{prop:main}).
This leads to a characterization of the curse of dimensionality for 
$\ell_p$-ball domains in the full regime $2\le p\le\infty$, thereby closing the gap 
in \cite{HNUW2}.

The key observation is that one needs to put the domain of integration in the right 
perspective to unveil and understand those geometric aspects 
that are crucial in 
such a problem. This is where the theory of Asymptotic Geometric Analysis enters the 
stage and provides us with a suitable and natural framework to settle this question. 
To the best of our knowledge, this is the first work revealing such a connection, which 
is promising and of independent interest. The central idea is to exploit the fact that 
volume is highly concentrated in high-dimensional (isotropic) convex bodies. 
There are a number of deep results in this direction, among others, the pioneering works of Parouris \cite{P2006} and Klartag \cite{K2007}. 
The main ingredient in our proof is the famous thin-shell estimate due 
to Gu\'edon and Milman \cite{GM2011}, which shows that the Euclidean norm of an isotropic 
and log-concave random vector $X\in\R^d$ is highly concentrated around its expectation. 
The reason for using their estimate, despite the fact that it is the best known general 
bound, is the sensitivity of our estimates to the involved deviation parameters (see Remark \ref{rem:deviation parameters}). 
Under very natural geometric assumptions on the domain, reflecting a typical framework in
Asymptotic Geometric Analysis, we are able to relax the radial 
condition~\eqref{eq:small radius} from \cite{HNUW2}. 
As will become apparent later on, condition \eqref{eq:small radius} uses, 
although in disguise, the well-known universal lower bound $1/\sqrt{2\pi e}$ for the isotropic constant. But as we shall demonstrate, the characterization of the curse of dimensionality does depend (to some extend) on the isotropic constant of the domain of integration (see also Remark \ref{rem:condition}), which is probably of independent interest again.

\subsection{Main results}
We now specify the geometric framework we shall be working in, followed by a presentation of the main results of this paper. More information on the central (geometric) notions that appear here shall be provided in the preliminaries.

A convex body $K_d\subseteq\R^d$ is a compact, convex set with non-empty interior and shall be called isotropic if $\vol_d(K_d)=1$, the center of mass is at the origin and, 
for some $L_{K_d}\in(0,\infty)$, 
\[
\int_{K_d} \langle x, \theta\rangle^2\dint x \,=\, L_{K_d}^2
\]
for all $\theta\in\SSS^{d-1}$. We call $L_{K_d}$ the isotropic constant of 
$K_d$. A convex body $K_d\subseteq\R^d$ is symmetric if $-x\in K_d$ whenever $x\in K_d$. The radius of such a symmetric convex body is defined as
\[
\rad(K_d) = \sup_{y\in K_d} \|y\|_2.
\]
We say that an isotropic convex body $K_d$ satisfies a 
$\psi_\alpha$-estimate with constant $b_\alpha\in(0,\infty)$, for some $\alpha\in[1,2]$, if 
\[
\|\langle\cdot,\theta\rangle\|_{\psi_\alpha} 
\,\le\, b_\alpha \left(\int_{K_d} |\langle x, \theta\rangle|^2\dint x \right)^{1/2}
\]
for all $\theta\in\SSS^{d-1}$. The Orlicz norm $\|\cdot\|_{\psi_\alpha}$ 
will be defined in Section~\ref{sec:BackgroundB}. For short, we call an isotropic convex body with this property a $\psi_\alpha$-body. 

The first main result is the following. 

\begin{thrm}\label{thm:main}
Let $(K_d)_{d\in\N}$ be a sequence of symmetric, isotropic convex bodies that 
satisfy a uniform $\psi_2$-estimate
and 
\begin{align}\label{eq:small-diameter}
\limsup_{d\to\infty}\, \frac{\rad(K_d)}{\sqrt{d}\,L_{K_d}} \,<\, 2\,.
\end{align}
Then the curse of dimensionality 
holds for the class $\mathscr C^1_d(A_d,B_d,K_d)$ 
if 
\[
\limsup_{d\to\infty}\;\; \min\left\{ A_{d} \, \sqrt{d} \,L_{K_d},\; 
B_{d} \, d \,L_{K_d}^2\right\} \,>\, 0\,.
\] 
\end{thrm}
\bigskip

In addition to these lower bounds on the complexity of numerical integration, 
we will present lower bounds for the more general class of 
$\psi_\alpha$-bodies with $\alpha\in[1,2)$ (see Section~\ref{sec:proof}). 
These bounds are, however, no longer exponential in $d$, but in $d^{\alpha/2}$. 
We shall present this result in Proposition~\ref{prop:main}.

\begin{remark}
In Asymptotic Geometric Analysis terms, our radial condition 
\eqref{eq:small-diameter} simply means that such a body $K_d$ is of 
``small diameter'' (i.e., $\rad(K_d) = \alpha \sqrt{d}\,L_{K_d}$ for some $\alpha\in[1,\infty)$) with a constant strictly smaller than $2$. It follows directly from the isotropic condition that, for any isotropic convex body $K\subseteq \R^d$, $\sqrt{d}\,L_K \leq \rad(K)$. Additionally, as shown by Kannan, Lov\'asz and Simonovits in \cite{KLS1995}, we have the upper bound $\rad(K)\le(d+1)L_K$. We refer to the monograph \cite{IsotropicConvexBodies} for more information.
\end{remark}

\begin{remark}\label{rem:condition}
If we compare the two conditions \eqref{eq:small radius} and 
\eqref{eq:small-diameter} on the radii, then we see that, contrary to 
condition \eqref{eq:small radius} considered in \cite{HNUW2}, 
our condition \eqref{eq:small-diameter} involves the isotropic constant $L_{K_d}$ of the integral domain $K_d$. 
Using the well-known fact that among all isotropic convex bodies the 
(volume normalized) Euclidean ball, we write $\DD_2^d$, minimizes the isotropic 
constant with $L_{\DD_2^d}\geq 1/\sqrt{2\pi e}$ (see, e.g., \cite[Proposition 3.3.1]{IsotropicConvexBodies}), we obtain from our 
assumption \eqref{eq:small-diameter} exactly the radius condition 
\eqref{eq:small radius}. The observation that this universal lower bound on 
the isotropic constant of the convex body under consideration, which appears 
in \eqref{eq:small radius} only in disguise, can be substituted by the 
isotropic constant of the body itself, is the essential step that allows us 
to extend the results from \cite{HNUW2}. At the same time it sheds light on the r\^ole of isotropicity in this problem.
\end{remark}

\begin{remark}
 As will become apparent, Theorem~\ref{thm:main} is essentially Proposition~\ref{prop:main} below in the case $\alpha=2$. As a matter of fact, in the statement of Theorem~\ref{thm:main} the isotropic constants appearing in the two conditions involving the Lipschitz parameters $(A_d)_d$ and $(B_d)_d$ could have been omitted. The reason is that Bourgain proved in \cite{B2003} that if $K$ is a symmetric $\psi_2$-body with constant $b\geq 1$, then its isotropic constant is bounded, even more precisely, $L_K\leq Cb\log(b+1)$ with an absolute constant $C\in(0,\infty)$. However, we chose to include the constants $(L_{K_d})_d$ anyways to underline where the isotropicity enters the scene. Whether or not the isotropic constant is uniformly bounded above by an absolute constant in general is a famous open problem (see Remark \ref{rem:LK}).
\end{remark}

\medskip

We complement our lower bounds by the following concentration result, 
which easily implies an upper bound on the complexity of numerical integration 
in the space $\mathscr C^1_d(A_d,B_d,K_d)$. 
We stress that this result, as the last one, heavily relies on the 
thin-shell estimate of Gu\'edon and Milman \cite{GM2011}.

\begin{thrm}\label{thm:main-upper}
Let $(K_d)_{d\in\N}$ be a sequence of isotropic convex bodies. 
Assume that $(A_d)_{d\in\N}$ and $(B_d)_{d\in\N}$ satisfy
\[
\lim_{d\to\infty}\; \min\left\{ A_{d} \, \sqrt{d} \,L_{K_d},\; 
B_{d} \, d \,L_{K_d}^2\right\} \,=\, 0\,.
\]
Then 
\[
\lim_{d\to\infty}\, \sup_{f\in \mathscr C^1_d(A_d,B_d,K_d)}\, 
\left| \int_{K_d} f(x)\dint x \,-\, f(0) \right| \,=\, 0\,.
\]
In particular, the curse of dimensionality does not hold for $\mathscr C^1_d(A_d,B_d,K_d)$.
\end{thrm}
\smallskip

This theorem generalizes the corresponding result \cite[Proposition 4.7]{HNUW2} to 
arbitrary isotropic sets.  
Moreover, the necessary decay of the Lipschitz constants is the same as in 
\cite[Theorem~4.1]{HNUW2} if the hyperplane conjecture were shown to be true (see Remark~\ref{rem:LK}).
\smallskip

The volume-normalized $\ell_p$-balls, $\DD_p^d$, are symmetric and isotropic 
$\psi_2$-bodies whenever $2\leq p \leq \infty$ (see \cite[Proposition 10]{BGMN2005}). 
Moreover, 
we will show in Lemma~\ref{lem:lp-diameter} that those $\DD_p^d$-balls 
satisfy \eqref{eq:small-diameter}. As a consequence, we obtain the 
curse of dimensionality for this important class of integral domains as a 
corollary to Theorem~\ref{thm:main}.

\begin{cor}\label{cor:lp}
Let $2\leq p \leq \infty$. Then the curse of dimensionality holds for the 
classes $\mathscr C^1_d(A_d,B_d,\DD_p^d)$ if and only if 
\[
\limsup_{d\to\infty}\;\; \min\left\{  \sqrt{d}\, A_{d},\; 
d\, B_{d}\right\} \,>\, 0\,.
\] 
\end{cor}

\bigskip

\begin{remark}
We note that there is a mistake in the formulation of the
corresponding result from~\cite{HNUW2}. The assumptions on $L_{0,d}$ and $L_{1,d}$ 
in~\cite[Theorem~4.1]{HNUW2} (which we call $A_d$ and $B_d$) 
are not enough and should be replaced by an assumption as above.
\end{remark}

\medskip

The rest of the paper is organized as follows. 
In the next section, Section \ref{Sec:prelim and notation}, 
we provide background material from both Asymptotic Geometric Analysis and 
Information-based Complexity. Having a broad readership in mind, 
we do this in slightly more detail than strictly necessary for the proof of our results.
In Section~\ref{sec:aux}, we present some auxiliary results in form of radial and volume estimates. The final Sections~\ref{sec:proof}--\ref{sec:proof-upper} 
are then devoted to the proofs of the main results mentioned above.\\

\medskip
\section{Preliminaries and notation}\label{Sec:prelim and notation}

In this section we will present the notions and concepts from Information-based Complexity and Asymptotic Geometric Analysis needed throughout this work. Keeping a broad readership in mind, we try to keep this as detailed and self-contained as possible.

\subsection{Background A -- Information-based Complexity} \label{sec:IBC}

Information-based Complexity (IBC) studies optimal algorithms and computational complexity for continuous problems like finding solutions of differential equations, integration and approximation, arising in different areas of application. The emphasis is on the dependence of the minimal error achievable within a certain class of algorithms using a given budget of $n$ informations about, e.g., the function to integrate or approximate. Typically, the functions under consideration depend on many variables, so also the dependence on the number $d$ of variables, the dimension of the problem, is crucial.
For recent monographs thoroughly treating many modern results in IBC, we refer the reader to \cite{NW08,NW10,NW12}.

In this subsection, we recall the necessary notions from IBC to precisely define our problem and explain notation already used in formulating the main results in the introduction.

\medskip
\subsubsection{The setting}

Let $\mathscr F_d$  be a class 
of continuous and integrable functions $f:K_d\to\R$, where $K_d\subseteq \R^d$ is measurable and of unit volume. For $f\in \mathscr F_d$,
we want to approximate the integral  
$$
S_d(f) = \int_{K_d} f(x)  \dint x
$$
by algorithms
$$
A_{n,d}(f)=\phi_{n,d}\big(f(x_1),f(x_2),\dots,f(x_n)\big),  
$$
where $\phi_{n,d}:\R^n\to \R$ is an arbitrary mapping and $x_j\in K_d$, $j\in\{1,\dots,n\}$ can be chosen adaptively. Adaptively means that the selection of $x_j$ may depend on the already computed values $f(x_1),f(x_2),\dots,f(x_{j-1})$.
The (worst case) error of approximation of the algorithm $A_{n,d}$ is defined as
$$
e(A_{n,d}):=\sup_{f\in \mathscr F_d}\big|S_d(f)-A_{n,d}(f)\big|.
$$

\medskip
\subsubsection{Complexity}

For $\varepsilon>0$, the information complexity, $n(\eps,\mathscr F_d)$, is 
the minimal number of function values needed to guarantee that the error is
at most $\eps$, i.e., 
$$
n(\eps,\mathscr F_d):=\min\big\{\,n\in\N \, :\,  \exists\ A_{n,d}\ \ \mbox{such that}\ \ 
e(A_{n,d})\le\eps\big\}.
$$
Hence, we minimize $n$ over all possible choices of adaptive sample points $x_1,\dots,x_n\in K_d$ and mappings $\phi_{n,d}:\R^n\to \R$.
\begin{remark}
It is known by the result of Smolyak \cite{S65} on nonlinear algorithms and the result of Bahvalov \cite{B71} on adaption that,
as long as the class $\mathscr F_d$
is convex and symmetric, we may restrict the minimization of $n$ by
considering only \emph{nonadaptive} choices of $x_1,\dots, x_n$ and \emph{linear} mappings $\phi_{n,d}$ (see also \cite{NW08,TWW88}).
In this case, we have 
\begin{equation}\label{eq:complexity}
n(\eps,\mathscr F_d)=\min\Bigg\{\,n\in\N \,:\, \inf_{x_1,\dots,x_n\in K_d}\ \sup_{f\in \mathscr F_d,\,
f(x_j)=0\atop{\,1\leq j\leq n}} |S_d(f)|\le \eps\Bigg\}, 
\end{equation}
see, e.g., \cite[Lemma 4.3]{NW08}.
\end{remark} 

In this paper, we always consider convex and symmetric $\mathscr F_d$ 
so that we can use the previous formula for the information complexity $n(\eps,\mathscr F_d)$. 
It is also well known that for convex and symmetric $\mathscr F_d$ the total
complexity, i.e., the minimal cost of computing an $\eps$-approximation, 
insignificantly differs from the information
complexity.
For more details see, for instance, \cite[Section 4.2.2]{NW08}.

\medskip
\subsubsection{The curse of dimensionality}
By the curse of dimensionality we mean that the information complexity 
$n(\eps,\mathscr F_d)$ is exponentially large in $d$. 
That is, there are positive numbers $c$, $\eps_0$ and $\gamma$ such that
\begin{equation}\label{curse}
n(\eps,\mathscr F_d) \ge c \, (1+\gamma)^d\ \ \ \ 
\mbox{for all}\ \ \ \eps \le \eps_0\ \  \mbox{and infinitely many}\ \ d\in \N. 
\end{equation}
There are many classes $\mathscr F_d$ for which the curse of dimensionality has 
been proved for numerical integration and other multivariate problems 
(see~\cite{NW08,NW10,NW12} for such examples). From a computational point of view, the curse of dimensionality renders the problem intractable in high dimensions. 

\medskip
\subsection{Background B -- Asymptotic Geometric Analysis} \label{sec:BackgroundB}

We present here the background material from Asymptotic Geometric Analysis, 
which we organize by subtopic.  

\subsubsection{Convex bodies}

We shall be working in $\R^d$ equipped with the standard Euclidean structure $\langle \cdot,\cdot\rangle$ and use the notation $\vol_d(A)$ to indicate the $d$-dimensional Lebesgue measure of a Borel subset $A$ of $\R^d$. For $A\subseteq \R^d$, we define its radius to be
\[
\rad(A) := \inf_{x\in\R^d} \sup_{y\in A} \|y-x\|_2.
\]
For symmetric convex sets $A\subseteq\R^d$, that is, sets that are convex and for which $-x\in A$ whenever $x\in A$, this simplifies to
\[
\rad(A) = \sup_{y\in A} \|y\|_2.
\]
A convex body $K\subseteq \R^d$ is a compact and convex set with non-empty interior. We write $\SSS^{d-1} = \{ x \in \R^d : \|x\|_2 = 1\}$ for the Euclidean unit sphere in $\R^d$ and $\sigma:=\sigma_{d-1}$ for the uniform probability measure on $\SSS^{d-1}$. A convex body is said to be isotropic (or in isotropic position) if $\vol_d(K)=1$, its center of mass is at the origin, i.e., 
\[
\int_Kx \dint x=0
\]
and it satisfies the isotropic condition, i.e., there exsits a constant $L_K\in(0,\infty)$ such that, for all $\theta\in\SSS^{d-1}$,
\[
\int_K\langle x,\theta\rangle^2 \dint x=L_K^2.
\]
We call $L_K$ the isotropic constant of $K$. 

\begin{remark}
There are several classical positions of convex bodies, e.g., John's position or the $M$-position, many of which arise as solutions to extremal problems. The isotropic position first arose from classical mechanics in the $19$th century and has different research directions connected with it, one being the distribution of volume in convex bodies. In this work, we present yet another and new connection. For more information on the isotropic position we refer to \cite[Chapter 2]{AymptoticGeometricAnalysis}.
\end{remark}
\begin{remark}\label{rem:LK}
While the Euclidean ball minimizes the isotropic constant with $L_{\DD_2^d}\geq 1/\sqrt{2\pi e}$, whether or not the isotropic constant is uniformly bounded above by an absolute constant in general is a famous open problem first posed by Bourgain in \cite{B1991}. He obtained a general upper bound of order $\sqrt[4]{d}\log d$, the proof being based on the $\psi_1$-behavior of linear functionals on convex bodies (see Subsection \ref{subsec:psi_alpha} for a definition). This was improved by Klartag to $\sqrt[4]{d}$ in \cite{K2006}. Interestingly, while the problem remains open, there is no example of a convex body $K$ with $L_K>1$.
\end{remark}

\smallskip
\subsubsection{Isotropic log-concave probability measures}

We say that a Borel probability measure $\mu$ on $\R^d$, which is absolutely continuous with respect to the Lebesgue measure, is centered if for all $\theta\in \SSS^{d-1}$,
\[
\int_{\R^d} \langle x,\theta \rangle \dint\mu(x) = 0.
\]
We say that $\mu$ is log-concave if, for all compact subsets $A$ and $B$ of $\R^d$ and all $\lambda\in(0,1)$,
\[
\mu\big((1-\lambda)A + \lambda B\big) \geq \mu(A)^{1-\lambda}\mu(B)^\lambda.
\] 
A function $f:\R^d \to [0,\infty)$ is said to be log-concave if it satisfies 
\[
f\big((1-\lambda)x + \lambda y\big) \geq f(x)^{1-\lambda}f(y)^\lambda
\] 
for all $x,y\in\R^d$, $\lambda\in(0,1)$. We call it a log-concave density if additionally $\int_{\R^d} f(x) \dint x=1$. 

It was shown by Borell in \cite{Borell75} that any non-degenerate log-concave probability measure $\mu$ on $\R^d$ (meaning it is not fully supported on any hyperplane) is absolutely continuous with respect to the Lebesgue measure and has a log-concave density $f_\mu$, i.e., $\dint\mu(x) = f_\mu(x)\dint x$ (see also \cite[Theorem 2.1.2]{IsotropicConvexBodies}). Let $(\Omega,\mathcal A,\Pro)$ be a probability space. A random vector $X:\Omega\to\R^d$ will be called log-concave if its distribution $\mu(\cdot)=\Pro(X\in \cdot)$ is a log-concave probability measure on $\R^d$. 

A Borel probability measure $\mu$ on $\R^d$, which is absolutely continuous with respect to the Lebesgue measure, is said to be isotropic if it is centered and satisfies the isotropic condition
\[
\int_{\R^d}\langle x,\theta \rangle^2\dint\mu(x) = 1,
\] 
for all $\theta\in \SSS^{d-1}$. 
A log-concave random vector $X$ in $\R^d$ is said to be isotropic if its distribution is isotropic, i.e., if $\E(X)=0$ and $\E(X\otimes X)=\id_d$, where $\id_d$ denotes the $d\times d$ identity matrix.

\begin{remark}\label{examples isotropic log-concave}
Notice that a convex body is isotropic if and only if the uniform probability measure on $\frac{K}{L_K}$ is isotropic. Examples of isotropic log-concave random vectors are standard Gaussian random vectors  or  random  vectors  uniformly  distributed  in $\frac{K}{L_K}$, where $K$ is an isotropic convex body and $L_K$ its isotropic constant.
\end{remark}

\medskip
\subsubsection{$\psi_\alpha$-estimates for linear functionals} \label{subsec:psi_alpha}

Let $\alpha\in[1,2]$ and let $\mu$ be a probability measure on $\R^d$. For a measurable function $f:\R^d\to\R$ define the Orlicz norm $\|\cdot\|_{\psi_\alpha(\mu)}$ by
\[
\|f\|_{\psi_\alpha(\mu)} := \inf\bigg\{ \lambda>0 : \int_{\R^d} e^{|f(x)/\lambda|^\alpha} \dint\mu(x) \leq 2 \bigg\}.
\]
Given some $\theta\in \SSS^{d-1}$ one says that $\theta$ defines a $\psi_\alpha$-direction for $\mu$ with constant $C\in(0,\infty)$ if $f_\theta(x)=\langle x,\theta \rangle$ satisfies
\[
\|f_\theta\|_{\psi_\alpha(\mu)} \leq C \left( \int_{\R^d} |f_\theta(x)|^2\dint\mu(x) \right)^{1/2}.
\]


Let $K\subseteq \R^d$ be a convex body with centroid at the origin. Such a body is said to be a $\psi_\alpha$-body with constant $b_\alpha\in(0,\infty)$ if all directions $\theta\in \SSS^{d-1}$ are $\psi_\alpha$ with constant $b_\alpha$, with respect to the uniform probability measure on $K$. 
\begin{remark}
Every isotropic convex body $K_d$ or, more generally, each log-concave probability measure on $\R^d$, satisfies a $\psi_1$-estimate with some universal constant $C\in(0,\infty)$ that is independent of $d$ and $K_d$ (see, e.g., \cite[Section 3.2.3]{IsotropicConvexBodies}).
\end{remark}

Regarding the unit balls of finite-dimensional $\ell_p$ spaces, the following result was obtained by Barthe, Gu\'edon, Mendelson, and Naor \cite[Proposition 10]{BGMN2005}.

\begin{prop}\label{prop:psi2 bodies}
There exists $C\in(0,\infty)$ such that for every $d\in\N$ and every $p\geq2$, $\BB_p^d$ is a $\psi_2$-body with constant $C$.
\end{prop}


\medskip
\subsubsection{Thin-shell estimates}

One of the driving forces of the theory of Asymptotic Geometric Analysis is the isotropic constant or hyperplane conjecture and with it the question of how volume is distributed in high-dimensional isotropic convex bodies. The last decade has seen major contributions in this direction. Among the most important ones are Paouris' result on the tail behavior of the Euclidean norm of an isotropic log-concave random vector \cite{P2006}, Klartag's thin-shell estimate that resolved the central limit problem for log-concave measures \cite{K2007}, and the concentration results for the Euclidean norm of an isotropic log-concave random vector due to Gu\'edon and Milman \cite{GM2011}. The latter result, more precisely \cite[Theorem 1.1]{GM2011}, plays a crucial r\^ole in our proofs. Denoting by $\|\cdot\|_{\textrm{HS}}$ the Hilbert-Schmidt norm and by $\|\cdot\|_{\textrm{op}}$ the operator norm, their result reads as follows.

\begin{thrm}\label{thm:GM11}
Let $X$ be an isotropic random vector in $\R^d$ with log-concave density, which 
is in addition $\psi_{\alpha}$, for some $\alpha\in[1,2]$, with constant 
$b_\alpha\in(0,\infty)$. Assume that $A\in \R^{d\times d}$ satisfies $\|A\|_{\textrm{HS}}^2 = d$. 
Then, for all $t\geq 0$,
\[
\Pro\Big( \big|\|AX\|_2-\sqrt{d}\,\big| \geq t\sqrt{d}\,\Big) \leq C \exp\big(-c\eta^{\frac{\alpha}{2}}\min\{t^{2+\alpha},t \}\big) ,
\]
where
\[
\eta:= \frac{d}{\|A\|_{\textrm{op}}^2\,b_\alpha^2}\,.
\]
\end{thrm}

This result combined with Paouris' theorem (see \cite[Theorem 1.3]{Pao2012}), gives the following deviation and small-ball estimate, respectively:
\vskip 1mm
\noindent For all $t\geq 0$,
\begin{align}\label{eq:large devi}
\Pro\Big( \|AX\|_2 \geq (1+t)\sqrt{d}\Big) \leq \exp\big(-c_1\eta^{\frac{\alpha}{2}}\min\{t^{2+\alpha},t \}\big),
\end{align}
and, for all $t\in [0,1]$, 
\begin{align}\label{eq:small ball}
\Pro\Big( \|AX\|_2 \leq (1-t)\sqrt{d}\Big) \leq C \exp\Big(-c_2\eta^{\frac{\alpha}{2}}\max\{t^{2+\alpha},\log\frac{c_3}{1-t} \}\Big),
\end{align}
where $c_1,c_2,c_3,C\in(0,\infty)$ are absolute constants.

\begin{remark}\label{rem:deviation parameters}
It will be essential later that \eqref{eq:small ball} holds for \emph{any} deviation parameter $t\in[0,1]$. While Paouris' small-ball estimate from \cite{P2006} is of the same flavor, it only holds for $t\geq C>0$ for some absolute constant $C\in(0,\infty)$, while it is crucial in our proofs to take $t$ arbitrarily small. 
\end{remark}

\medskip
\subsubsection{Geometry of $\ell_p$-balls}\label{sec:lp}
Let $d\in\N$ and consider the $d$-dimensional space $\R^d$. For any $1\leq p\leq\infty$, the $\ell_p^d$-norm, $\|\cdot\|_p$, of a vector $x=(x_1,\ldots,x_d)\in\R^d$ is given by
\[
\|x\|_p := \begin{cases}
\Big(\sum\limits_{i=1}^d|x_i|^p\Big)^{1/p} &: p<\infty\\
\max\{|x_1|,\ldots,|x_d|\} &: p=\infty\,.
\end{cases}
\]
We will denote by $\BB_p^n:=\big\{x=(x_1,\dots,x_d)\in\R^d\,:\,\|x\|_p\leq 1\big\}$ the unit ball in $\ell_p^d$ and by $\DD_p^d$ the volume-normalized version, i.e., 
\[
\DD_p^d :=  \frac{\BB_p^d}{\vol_d(\BB_p^d)^{1/d}}\,.
\]
Recall that
\[
\vol_d(\BB_p^d) = \frac{2^d\Gamma(1+1/p)^d}{\Gamma(1+d/p)}\qquad\text{and}\qquad \rad(\BB_p^d)=d^{\max\{\frac{1}{2}-\frac{1}{p},0 \}}.
\] 
Thus, if we define 
\begin{equation}\label{eq:adp}
\alpha_{d,p}:= \vol_d(\BB_p^d)^{-1/d} = \frac{\Gamma(1+d/p)^{1/d}}{2\Gamma(1+1/p)},
\end{equation}
then we can write
\[
\DD_p^d = \left\{x\in\R^d : \|x\|_p \leq \alpha_{d,p} \right\}.
\]
The {cone (probability) measure} $\bM_{\BB_p^d}$ on $\BB_p^d$ is defined as
$$
\bM_{\BB_p^d}(B) = \frac{\vol_d\big(\{rx:x\in B\,,0\leq r\leq 1\}\big)}{\vol_d(\BB_p^d)}\,,
$$
where $B\subseteq\SSS_p^{d-1}$ is a Borel subset. We remark that $\bM_{\BB_p^d}$ coincides with the normalized surface measure on $\SSS_p^{d-1}$ if and only if $p=1$, $p=2$ or $p=\infty$. For a more detailed account to the relationship between the cone and the surface measure on $\ell_p$-balls we refer the reader to \cite{N,NR2002}.

We shall also use the following polar integration formula, stated here only for the case of $\ell_p^d$-balls:
	\begin{align}\label{eq:polar integration p-balls}
	\int_{\R^d}f(x)\,\dint x = d\,\vol_d(\BB_p^d)\int_0^\infty\int_{\SSS_p^{d-1}}f(ry)\,r^{d-1}\,\dint\bM_{\BB_p^d}(y)\dint r,
	\end{align}
where $f:\R^d\to\R$ is a non-negative measurable function (in fact, this may alternatively be used as a definition for the cone measure $\bM_{\BB_p^d}$ on $\BB_p^d$).

Let us rephrase the following result of Schechtman and Zinn \cite[Lemma 1]{SZ} (independently obtained by Rachev and R\"uschendorf in \cite{RR91}) that provides a probabilistic representation of the cone measure $\bM_{\BB_p^d}$ of the unit ball of $\ell_p^d$ (see also \cite{BGMN2005} for an extension) and shall be used later with the previous polar integration formula.

\begin{prop}\label{prop:SZ}
Let $d\in\N$, $1\leq p < \infty$, and $g_1,\dots,g_d$ be independent real-valued random variables that are distributed according to the density
\[
f(t) = \frac{e^{-|t|^p}}{2\Gamma\big(1+{1\over p}\big)}\,, \qquad t\in\R\,.
\]
Consider the random vector $G=(g_1,\dots,g_d)\in\R^d$ and put $Y:=G/\|G\|_p$. Then $Y$ is independent of $\|G\|_p$ and has distribution $\bM_{\BB^d_p}$.
\end{prop}

\bigskip
\section{Auxiliary computations} \label{sec:aux}

We present here some estimates that we need to prove the main results. In the first part, we show that the volume-normalized $\ell_p^d$-balls satisfy the radial condition \eqref{eq:small-diameter}. In the second part, we estimate the volume of the intersection of a dilated Euclidean ball with a $\psi_\alpha$-body.

\subsection{Radial estimates for $\ell_p$-balls}

To show that the radial assumption \eqref{eq:small-diameter} is satisfied, we need the following lemma.

\begin{lem}\label{lem:gdp}
Let $d\in\N$ and $1\leq p \leq \infty$. Then
\[
 L_{\DD_p^d}  \,=\, \alpha_{d,p}\cdot \gamma_{d,p}
\]
with $\alpha_{d,p}$ from~\eqref{eq:adp} and 
\[
\gamma_{d,p}^2:=\frac{1}{\vol_d(\BB_p^d)}\,\int_{\BB_p^d} x_1^2 \dint x =
\begin{cases}
\frac{p\,\Gamma(1+\frac{3}{p})\,\Gamma(1+\frac{d}{p})}{3(d+2)\,\Gamma(1+\frac{1}{p})\,\Gamma(\frac{d+2}{p})}& :\, 1\leq p <\infty \,;\\
\frac{1}{3}& :\, p=\infty \,.
\end{cases}
\]
In particular, for $p\geq 2$,
\[
\gamma_{d,p}^2 \geq \frac{d}{3(d+2)}\frac{\Gamma(1+\frac{3}{p})}{\Gamma(1+\frac{1}{p})}\,\Big(\frac{p}{d}\Big)^{2/p}.
\]
\end{lem}
\begin{proof}
From the definition of the isotropic constant (using it for the first standard unit vector, i.e., $\theta=e_1$) and a simple transformation, we obtain 
\[
 L_{\DD_p^d}^2  = \int_{\DD_p^d} x_1^2 \dint x 
 = \alpha_{d,p}^{d+2} \int_{\BB_p^d}x_1^2\dint x 
 = \alpha_{d,p}^2\cdot \gamma_{d,p}^2.
\]
Now, for $p=\infty$, we have
\[
 \gamma_{d,\infty}^2 = \frac{1}{\vol_d(\BB_\infty^d)}\,\int_{\BB_\infty^d} x_1^2 \dint x = \frac13 \, .
\]
Let $1\leq p<\infty$. Then, using the polar integration formula in \eqref{eq:polar integration p-balls} together with the probabilistic representation of the cone probability measure on $\BB_p^d$ (see Proposition \ref{prop:SZ}), we obtain
\[
\gamma_{d,p}^2=\frac{1}{\vol_d(\BB_p^d)}\,\int_{\BB_p^d} x_1^2 \dint x = \frac{p\,\Gamma(1+\frac{3}{p})\,\Gamma(1+\frac{d}{p})}{3(d+2)\,\Gamma(1+\frac{1}{p})\,\Gamma(\frac{d+2}{p})}\,.
\]
To prove the lower bound for $p\geq 2$, we use the inequality 
\begin{align}\label{eq:Gautschi}
   \frac{\Gamma(x+1)}{\Gamma(x+\lambda)} &\geq x^{1-\lambda} \quad \text{for } x>0 \text{ and } \lambda\in(0,1)\,,
\end{align}
which is due to Gautschi \cite{Gau1959}. Then, using \eqref{eq:Gautschi} with the choice $x=d/p$ and $\lambda=2/p$,
\begin{align*}
\gamma_{d,p}^2 & = \frac{p\,\Gamma(1+\frac{3}{p})\,\Gamma(1+\frac{d}{p})}{3(d+2)\,\Gamma(1+\frac{1}{p})\,\Gamma(\frac{d}{p}+\frac{2}{p})}
= \frac{d}{3(d+2)}\frac{p}{d} \frac{\Gamma(1+\frac{3}{p})\,\Gamma(1+\frac{d}{p})}{\Gamma(1+\frac{1}{p})\,\Gamma(\frac{d}{p}+\frac{2}{p})} \cr
& \geq \frac{d}{3(d+2)} \frac{p}{d} \frac{\Gamma(1+\frac{3}{p})}{\Gamma(1+\frac{1}{p})} \,\Big(\frac{p}{d}\Big)^{2/p-1}
= \frac{d}{3(d+2)}\frac{\Gamma(1+\frac{3}{p})}{\Gamma(1+\frac{1}{p})}\,\Big(\frac{p}{d}\Big)^{2/p}.
\end{align*}

\end{proof}

We remark that $\DD_p^d$ are symmetric and isotropic 
$\psi_2$-bodies whenever $2\leq p \leq \infty$ (see Proposition \ref{prop:psi2 bodies}). In view of Theorem~\ref{thm:main}, it is therefore left to prove that the sets $\DD_p^d$ satisfy assumption~\eqref{eq:small-diameter}.

\begin{lem}\label{lem:lp-diameter}
For $2\le p\le\infty$, we have
\[
\limsup_{d\to\infty} \, \frac{\rad(\DD_p^d)}{\sqrt{d}\, L_{\DD_p^d}} \,\le\, \sqrt{3} \, .
\]
\end{lem}
\begin{proof}
We only consider the case $2\le p < \infty$ as the case $p=\infty$ can be proved in a similar way but is easier.
%
From Lemma~\ref{lem:gdp}, we obtain that
\[
 \frac{\rad(\DD_p^d)}{\sqrt{d}\, L_{\DD_p^d}} 
\,=\,  \frac{\alpha_{d,p}\, d^{1/2-1/p}}{\sqrt{d}\, \alpha_{d,p}\, \gamma_{d,p}}
\,=\,  \frac{d^{-1/p}}{\gamma_{d,p}} 
\]
and 
\[
\limsup_{d\to\infty} \frac{\rad(\DD_p^d)}{\sqrt{d}\, L_{\DD_p^d}} 
\, = \, \limsup_{d\to\infty} \frac{d^{-1/p}}{\gamma_{d,p}}
\,\le\, \sqrt{3} \,  \sqrt{\frac{\Gamma(1+\frac{1}{p})}{\Gamma(1+\frac{3}{p})}p^{-2/p}}\,.
\]
Since $\lim_{p\to\infty}\frac{\Gamma(1+\frac{1}{p})}{\Gamma(1+\frac{3}{p})}p^{-2/p} = 1$, it is enough to show that the function
\[
 p \mapsto \frac{\Gamma(1+\frac{1}{p})}{\Gamma(1+\frac{3}{p})}p^{-2/p}
\] 
is increasing in $p$ for $p\ge 2$. Writing $x=\frac1p$ and taking natural logarithms, we have to show that
\[
 x \mapsto \ln \Gamma(1+3x) - \ln \Gamma(1+x) - 2 x \ln x
\]
is increasing in $x>0$. The derivative of this function is
\begin{align*}
 f(x) &:= 3 \psi(1+3x) - \psi(1+x) - 2 \ln x - 2 \\
      &\,= -2 \gamma -2 - 2 \ln x + 2 x \sum_{k=1}^\infty \frac{4k+3x}{k(k+x)(k+3x)}.
\end{align*}
Here $\psi$ is the Digamma function, which is the logarithmic derivative of the Gamma function, $\gamma$ is the Euler-Mascheroni constant ($\gamma\approx 0.577$), and the last identity follows from the well-known series expansion
\[
 \psi(1+x) = -\gamma + \sum_{k=1}^\infty  \frac{x}{k(k+x)} \,,
\]
which can be found in \cite{AbSt1992}.
Using the integral test, we obtain for $x>0$ that
\[
 2 x \sum_{k=1}^\infty \frac{4k+3x}{k(k+x)(k+3x)} > 2 x \int_1^\infty \frac{4k+3x}{k(k+x)(k+3x)} \dint k = 3 \ln(1+3x)-\ln(1+x),
\]
which shows that
\[ 
 f(x) > -2 \gamma -2 - 2 \ln x + 3 \ln(1+3x)-\ln(1+x). 
\]
To finish the proof, we finally show that $f(x)>0$ for $x>0$ by proving the estimate
\[
 g(x) := - 2 \ln x + 3 \ln(1+3x)-\ln(1+x) > 2 \gamma + 2.
\]
A direct computation shows that 
\[
g'(x) = -\frac2x + \frac{9}{1+3x}-\frac{1}{1+x}  = -\frac{2}{(1+3x)x(x+1)} <0
\]
for $x>0$. This implies that $g$ is decreasing and therefore,
\[
 g(x) > \lim_{t\to\infty} g(t) = \lim_{t\to\infty} \ln \frac{(1+3t)^3}{t^2(1+t)} = \ln 27 > 2 \gamma + 2  \, ,
\] 
since $\ln 27 \approx 3.296$ and $2\gamma +2 \approx 3.154$.
\end{proof}

\begin{remark}
It is easy to see that the limit-superior in Lemma \ref{lem:lp-diameter} is actually a limit.
\end{remark}

\medskip
\subsection{Volume estimates for intersections}

We apply Theorem \ref{thm:GM11} to obtain upper bounds on the volume of the intersection of $\ell_2$-balls with $\psi_\alpha$-bodies. This will be crucial in the proofs of the main results.

\begin{prop}\label{prop:volume}
Let $(K_d)_{d\in\N}$ be a sequence of  isotropic convex bodies that 
satisfy a uniform $\psi_\alpha$-estimate for some $\alpha\in[1,2]$. 
Moreover, let $(r_d)_{d\in\N}\in [0,\infty)^{\N}$ satisfy
\begin{equation}\label{eq:rd}
\limsup_{d\to\infty} \frac{r_d}{L_{K_d}} \,<\, 1.
\end{equation}
Then there exist absolute constants $c,C\in(0,\infty)$ such that
\[
\vol_d\Big(K_d\cap r_d\,\sqrt{d}\,\BB_2^d\Big) \leq C\, \exp\left(- c\, d^{\alpha/2} \right).
\]
\end{prop}

\begin{proof}
Let $X$ be a random vector chosen uniformly at random from $\frac{K}{L_K}$. 
The latter guarantees that the random vector is isotropic and log-concave (see Remark \ref{examples isotropic log-concave}).
We observe that, for any $\gamma\geq0$,
\[
\Pro\bigg(\|X\|_2 \leq \frac{r_d \sqrt{d}}{L_K}\bigg) \,=\, \vol_d\Big(K_d\cap r_d\,\sqrt{d}\,\BB_2^d\Big).
\]
From the assumptions, we have that $r_d\le a L_{K_d}$ for some constant 
$a\in(0,1)$ if $d$ is large enough.
Therefore, using the small-ball estimate \eqref{eq:small ball} with $A$ 
chosen to be the identity matrix, we obtain that
\[\begin{split}
\vol_d\Big(K_d\cap r_d\,\sqrt{d}\,\BB_2^d\Big) 
\,&\leq\, C \exp\left(- c_2\, \Big(\frac{d}{b_\alpha^2}\Big)^{\alpha/2} 
	\,\left(1-\frac{r_d}{L_{K_d}}\right)^{2+\alpha}\right) \\
\,&\le\, C \exp\left(- c\, d^{\alpha/2} \right),
\end{split}\]
for large enough $d$, 
where $c_2, C\in(0,\infty)$ are the absolute constants 
from~\eqref{eq:small ball} and $c:=c_2\,a^{2+\alpha}/b_\alpha^\alpha$. \\
\end{proof}

\medskip
\section{Proof of the main results} \label{sec:proof}

In this section we provide the proofs of our main results, Theorem \ref{thm:main} and Corollary~\ref{cor:lp}. 
Both will be direct corollaries of the following more general result.

\begin{prop}\label{prop:main}
Let $(K_d)_{d\in\N}$ be a sequence of symmetric and isotropic convex bodies that 
satisfy a uniform $\psi_\alpha$-estimate for some $\alpha\in[1,2]$  
and 
\begin{equation}\label{eq:condition}
\limsup_{d\to\infty}\, \frac{\rad(K_d)}{\sqrt{d}\,L_{K_d}} \,<\, 2.
\end{equation}
Moreover, let $(A_d)_{d\in\N}$ and $(B_d)_{d\in\N}$ be such that
\[
A_{d} \,\ge\, \frac{a}{\sqrt{d}\,L_{K_d}} 
\qquad \mbox{and} \qquad  B_{d} \,\ge\, \frac{b}{d\,L_{K_d}^2}
\qquad \mbox{for some } \quad a,b>0.
\] 
Then there exist constants $c,\gamma,\eps_0\in(0,\infty)$ such that, for all $\eps\in(0,\eps_0)$ 
and $d\in\N$, 
\[
n(\eps,\mathscr F_d) \,\ge\, c\, (1+\gamma)^{d^{\alpha/2}} 
\]
with $\mathscr F_d:=\mathscr C^1_d(A_d,B_d,K_d)$.
\end{prop}

\begin{proof}[Proof of Proposition~\ref{prop:main}]
The basic idea behind the proof is the same as in the 
proofs of the corresponding results from \cite{HNUW2}. 

For given sample points $x_1,\dots,x_n\in K_d$, 
we construct a fooling function that is defined on the entire 
$\R^d$ and satisfies the required bounds on the Lipschitz constants. 
This function will be zero at the points $x_1,\dots,x_n$ 
and will have a large integral in $K_d$ as long as $n$ is not 
exponentially large in $d$. 
Moreover, this function will be
zero on the entire convex hull of 
$x_1,\dots,x_n$.

To be precise, let $\calP_n$ be a set of $n$ points in the symmetric 
$\psi_\alpha$-body $K_d$ and let 
$\calC_n:=\conv(\calP_n)$ be their convex hull. 
Moreover, for a set $C\subseteq\R^d$ and $\delta\in(0,\infty)$ we define its $\delta\sqrt{d}$-extension to be
\[
C^{(\delta)} \,:=\, C+\delta\sqrt{d}\,\BB_2^d 
\,=\, \left\{x\in\R^d\colon \inf_{y\in C}\|x-y\|_2\le\delta\sqrt{d}\right\}.
\] 

By a result of Elekes~\cite{Ele1986}, which was adapted to this setting in 
\cite[Theorem 2.1]{HNUW2}, we know that, 
for arbitrary point sets $\calP_n\subseteq K_d$ of cardinality $n$, 
the convex hull $\calC_n$ is contained in the union of $n$ balls 
\[
y_1+\frac{\rad(K_d)}{2}\cdot\BB_2^d\,,\,\dots\,,\, y_n+\frac{\rad(K_d)}{2}\cdot\BB_2^d
\]
for some $y_1,\dots,y_n\in K_d$.
This implies that $\calC_n^{(\delta)}$ can be covered by 
$n$ balls with radii $\frac{\rad(K_d)}{2}+\delta\sqrt{d}$. 
Let $r_d:=\frac{\rad(K_d)}{2\sqrt{d}}+\delta$. 
We obtain 
\[
\vol_d\Big(\calC_n^{(\delta)}\cap K_d\Big)
\,\le\, \sum_{i=1}^n\vol_d\Big(K_d\cap \big(y_i+r_d\,\sqrt{d}\,\BB_2^d\,\big)\Big).
\]
Using the symmetry of $K_d$, which implies (by means of the dimension-free Brunn-Minkowski inequality) that the right hand side is 
maximized for $y_1=\dots=y_n=0$, we get
\[
\vol_d\Big(\calC_n^{(\delta)}\cap K_d\Big)
\,\le\, n\,\vol_d\Big(K_d\cap r_d\,\sqrt{d}\,\BB_2^d\Big).
\]
Together with Proposition~\ref{prop:volume}, this shows that
\begin{equation}\label{eq:volume-ch}
\vol_d\Big(\calC_n^{(\delta)}\cap K_d\Big) \,\le\, C\, n\, q^{d^{\alpha/2}}
\end{equation}
for some $q\in(0,1)$ if $r_d$ satisfies \eqref{eq:rd} and $d$ is large enough.

It is easy to check that $r_d$ satisfies \eqref{eq:rd} whenever
\begin{equation}\label{eq:delta}
\delta \,<\, L_{K_d}\,\left(1 - \frac{\rad(K_d)}{2\,\sqrt{d}\,L_{K_d}}\right).
\end{equation}
Clearly, such a $\delta\in(0,\infty)$ exists under the assumptions of 
Proposition~\ref{prop:main} for large enough $d\in\N$ 
since, in general, $L_{K_d}\ge1/\sqrt{2\pi e}$.

\medskip

We now construct, for an arbitrary point set $\calP_n\subseteq K_d$ of cardinality $n$, a function that fits our needs. 
For this we just use the distance function from the convex hull $\calC_n$ 
and smooth it out using a suitable piecewise polynomial. 
That is, we consider the function $f:K_d\to\R$ defined as
\begin{equation}\label{eq:function}
f(x) \,:=\, f_{\delta,\calP_n}(x)
\,=\, p_\delta\left(\inf_{y\in\calC_n}\|x-y\|_2\right),
\end{equation}
where
\[
p_\delta(t) \,=\, \begin{cases}
\frac{2}{\delta^2 d}\, t^2 & :\,  t\le \frac{\delta \sqrt{d}}{2}\,; \\
-\frac{2}{\delta^2 d}\, t^2 + \frac{4}{\delta \sqrt{d}}\, t-1 & :\,   
        t\in \bigl(\frac{\delta \sqrt{d}}{2},\delta \sqrt{d}\, \bigr)\,;\\
1, & :\,\quad  t\ge \delta \sqrt{d}\,.
\end{cases}
\]

By direct computations, we obtain that
\renewcommand\labelitemi{\tiny$\bullet$}
\begin{itemize}
        \item $f\in \mathscr C^1(\R^d)$, 
        \item $f(x)=0$ for $x\in \calC_n$,
        \item $f(x)=1$ for $x\notin \calC_n^{(\delta)}$,
        \item $\Lip(f) \le \frac{2}{\delta \sqrt{d}}$\; and 
        \item $\forall\theta\in\SSS^{d-1}$: $\Lip(D^\theta f) \le \frac{40}{\delta^2 d}$.
\end{itemize}
The precise computations can be found in \cite[Section~4.1]{HNUW2}, but note that 
we use here a function that is zero only on 
$\calC_n$ and not on $\calC_n^{(\delta)}$ (see~\cite[Remark 4.4]{HNUW2}).

The properties above and \eqref{eq:volume-ch} imply that
\[
\int_{K_d} f(x)\, \dint x \,\ge\, \vol_d\Big(K_d\setminus\calC_n^{(\delta)}\Big) \,=\,1-\vol_d\Big(K_d\cap\calC_n^{(\delta)}\Big)
\,\ge\, 1 - C\, n\, q^{d^{\alpha/2}} 
\]
with $C,q$ as in \eqref{eq:volume-ch} 
if $K_d$ is a $\psi_\alpha$-body with \eqref{eq:small-diameter} 
and $\delta\in(0,\infty)$ with~\eqref{eq:delta}.

As this bound holds for arbitrary point sets $\calP_n$ with $n$ elements, 
we see that the integral of the function $f$ constructed above must be larger than 
$\eps\in(0,1)$ as long as $n<C^{-1}q^{-d^{\alpha/2}}(1-\eps)$.
Taking into account the above properties of $f$ this shows
that the number of function evaluations 
that are necessary to guarantee an error of at most $\eps$ for all functions 
in $\mathscr F_{d,\delta}:=\mathscr C^1_d(A_d,B_d,K_d)$ 
with $A_d=\frac{2}{\delta \sqrt{d}}$ and $B_d =\frac{40}{\delta^2 d}$
satisfies
\[
n\big(\eps,\mathscr F_{d,\delta}\big) \,\ge\, c\, (1-\eps)\, (1+\gamma)^{d^{\alpha/2}}
\]
for some $c,\gamma\in (0,\infty)$ (see~\eqref{eq:complexity}) and $d$ large enough.
As this holds for all $\delta$ that satisfy~\eqref{eq:delta}, 
we obtain
\[
n\Big(\eps,\mathscr C^1_d(A_d,B_d,K_d)\Big) \,\ge\, c'\, (1+\gamma)^{d^{\alpha/2}}
\]
for some $c',\gamma\in (0,\infty)$, large enough $d$ and all $\eps\in(0,1/2)$, 
if $A_d\ge\frac{a_0}{\sqrt{d} L_{K_d}}$ and $B_d \ge\frac{b_0}{d\,L_{K_d}^2}$ 
for some $a_0,b_0>0$ that depend only on the precise value of the left hand side 
of~\eqref{eq:condition}.

The proof for arbitrary sequences $(A_d)$ and $(B_d)$ 
that satisfy the assumptions of the theorem
follows from a scaling as in the proof of \cite[Proposition 3.2]{HNUW2}.
For this, note that there exists a constant $M\in[1,\infty)$ such that
$M\cdot A_d\ge\frac{a_0}{\sqrt{d} L_{K_d}}$ and 
$M\cdot B_d \ge\frac{b_0}{d\,L_{K_d}^2}$
with $a_0, b_0$ from above. Therefore, 
\[
n\Big(\eps,\mathscr C^1_d(M\cdot A_d,M\cdot B_d,K_d)\Big) \,\ge\, c'\, (1+\gamma)^{d^{\alpha/2}}
\]
for some $c',\gamma\in (0,\infty)$, large enough $d$ and all $\eps\in(0,1/2)$.
Using
\[
\mathscr C^1_d(M\cdot A_d,M\cdot B_d,K_d)\subseteq M\cdot\mathscr C^1_d(A_d,B_d,K_d)
\]
and $n\big(\eps,M\cdot\mathscr F_d\big)=n\big(M\cdot\eps,\mathscr F_d\big)$, 
this implies
\[
n\Big(\eps,\mathscr C^1_d(A_d,B_d,K_d)\Big) \,\ge\, c'\, (1+\gamma)^{d^{\alpha/2}}
\]
for all $\eps\in(0,1/(2M))$ and large enough $d$. 
Clearly, we can now modify the constant $c'$ such that this lower bound holds 
for all $d\in\N$.
This proves the theorem.

\end{proof}

\begin{remark}\label{rem:delta-lp}
In the case of volume-normalized $\ell_p$-balls, the $\delta\in(0,\infty)$ that was used in the construction of the fooling function, 
see~\eqref{eq:function}, can be chosen to be
\[
\delta \,=\, \frac{1}{42} \,<\, \frac1{\sqrt{2\pi e}}\left(1 - \frac{\sqrt{3}}{2}\right)
\,\approx\, 0.0324
\]
for large enough $d$ (see Lemma~\ref{lem:lp-diameter}).
\end{remark}

\medskip

We finish this section with the proof of Theorem~\ref{thm:main} and Corollary~\ref{cor:lp}.

\begin{proof}[Proof of Theorem~\ref{thm:main}]
Theorem~\ref{thm:main} is just Proposition~\ref{prop:main} in the 
case $\alpha=2$. Note that the definition of the curse of dimensionality requires a lower 
bound as in Proposition~\ref{prop:main} only for infinitely many $d\in\N$.
\end{proof}

We now turn to the important class of $\ell_p$-balls. 
Based upon this we will then be able to close the aforementioned gap left in 
\cite{HNUW2} and prove the curse of dimensionality for numerical integration 
of smooth functions on $\ell_p$-balls in the full regime $2\leq p \leq \infty$.

\begin{proof}[Proof of Corollary~\ref{cor:lp}]
For the proof of the corollary it is enough to show that the sets $\DD_p^d$ 
with $2\le p\le\infty$ satisfy the conditions of Theorem~\ref{thm:main} and 
Theorem~\ref{thm:main-upper}. 
This follows from the considerations from Section~\ref{sec:lp} and 
Lemma~\ref{lem:lp-diameter}, see also Remark~\ref{rem:delta-lp}.
\end{proof}

\goodbreak

\section{Proof of Theorem \ref{thm:main-upper}} \label{sec:proof-upper}

The upper bounds on the complexity in the present setting 
(Theorem~\ref{thm:main-upper})
can be proved in a very similar way as \cite[Theorem~4.1]{HNUW2}. 
However, as the emphasis in \cite{HNUW2} was only on convex sets with 
small diameter, the authors did not find the precise dependence of the 
involved conditions on the isotropic constant.
Since the proofs are quite short, we repeat them here. 

For the statement of the following results, we define
\[
I_q(K_d) \,:=\, \int_{K_d} \|x\|_2^q \dint x
\]
for measurable sets $K_d\subseteq \R^d$ and $q\ge0$. 

We start with the condition on the Lipschitz constant of the function. 
Afterwards we state the result on the Lipschitz constant of the derivatives.

\begin{prop}\label{prop:upper1}
Let $(K_d)_{d\in\N}$ be a sequence of measurable sets with $\vol_d(K_d)=1$. 
Assume that
\[
\lim_{d\to\infty}\, A_{d}\cdot I_1(K_d) \,=\, 0.
\]
Then
\[
\lim_{d\to\infty}\, \sup_{f}\, 
\left| \int_{K_d} f(x)\dint x \,-\, f(0) \right| \,=\, 0,
\]
where the supremum is taken over all 
$f$ in the set $\big\{f\in \mathscr C(K_d)\colon \Lip(f)\le A_{d}\big\}$.
\end{prop}

\medskip

The next proposition deals with the optimality of the assumptions on the decay of 
$(B_d)_{d\in\N}$, i.e., the Lipschitz constants of the directional derivatives. 
However, in this case we have to assume more from the sets under consideration. 
Note that a set is called 
\emph{star-shaped with respect to the origin} if, for every $x$ from the set, 
the line segment from the origin to $x$ is also contained in the set. 
Clearly, this is satisfied by every convex set that contains the origin.

\begin{prop}\label{prop:upper2}
Let $(K_d)_{d\in\N}$ be a sequence of measurable sets with $\vol_d(K_d)=1$, 
center of mass at the origin and star-shaped with respect to the origin. Assume that
\[
\lim_{d\to\infty}\, B_{d}\cdot I_2(K_d) \,=\, 0.
\]
Then
\[
\lim_{d\to\infty}\, \sup_{f}\, 
\left| \int_{K_d} f(x)\dint x \,-\, f(0) \right| \,=\, 0,
\]
where the supremum is taken over all $f$ in the set 
\[ 
\big\{f\in \mathscr C^1(K_d)\,:\,  \Lip(D^\theta f)\le B_{d} \text{ for all } \theta\in\SSS^{d-1}\big\}.
\]
\end{prop}

\bigskip

\begin{proof}[Proof of Proposition~\ref{prop:upper1}]
Using the estimate
\[
|f(x)-f(0)| \,\le\, A_d\,\|x\|_2,
\]
valid for all $x\in K_d$ and 
$f\in \mathscr C(K_d)$ with $\Lip(f)\le A_{d}$, 
together with the triangle inequality, we obtain 
\begin{equation}\label{eq:proof1}
\left| \int_{K_d} f(x)\dint x \,-\, f(0) \right|
\,\le\, A_d\,\int_{K_d}\|x\|_2\dint x
\,=\, A_d \cdot I_1(K_d).
\end{equation}
\end{proof}

\begin{proof}[Proof of Proposition~\ref{prop:upper2}]
Similar to \cite[Proposition 4.7]{HNUW2}, we use that the mean value 
theorem implies
\[
f(x)-f(0) \,=\, \langle \nabla f(y_x), x\rangle
\]
for some $y_x$ on the line segment between $0$ and $x$, and $f\in \mathscr C^1(K_d)$. 
Moreover, since the center of mass of $K_d$ is at the origin, we have 
$\int_{K_d} \langle \nabla f(0), x \rangle\, \dint x = 0$. Hence,
\[\begin{split}
\int_{K_d} f(x) \dint x - f(0)
\;&=\; \int_{K_d} \big(f(x) - f(0) - \langle \nabla f(0), x\rangle \big) \, \dint x \\
&=\; \int_{K_d} \big\langle \nabla f(y_x) - \nabla f(0), x \big\rangle \, \dint x. 
\end{split}\]
With $\theta_x:=x/\|x\|_2$, we obtain 
$\langle \nabla f(y), x \rangle = D^{\theta_x}f(y)\, \|x\|_2$ 
and 
\[\begin{split}
\abs{\big\langle \nabla f(y_x) - \nabla f(0), x\big\rangle}
\;&=\; \abs{D^{\theta_x}f(y_x)-D^{\theta_x}f(0)}\, \|x\|_2 \\
&\le\; \Lip\big(D^{\theta_x}f\big)\, \|y_x\|_2\, \|x\|_2 \\
&\le\; \Lip\big(D^{\theta_x}f\big)\, \|x\|_2^2\,.
\end{split}\]
If now $\Lip(D^\theta f)\le B_{d}$ for all $\theta\in\SSS^{d-1}$, we get
\begin{equation} \label{eq:proof2}
\abs{\int_{K_d} f(x) \dint x - f(0)}
\,\le\, B_d\, \int_{K_d} \|x\|_2^2\, \dint x
\,=\, B_d \cdot I_2(K_d).
\end{equation}
\end{proof}
\medskip

Obviously, Proposition~\ref{prop:upper1} and Proposition~\ref{prop:upper2} 
hold also if the corresponding classes of functions are replaced by 
$\mathscr C^1_d(A_d,B_d,K_d)$. Moreover, it is clear from the definition that 
$I_2(K_d) = d\,L_{K_d}^2$ and 
$I_1(K_d) \le \sqrt{I_2(K_d)} = \sqrt{d}\,L_{K_d}$ 
for isotropic $K_d$.
Therefore, combining \eqref{eq:proof1} and \eqref{eq:proof2}
proves the corresponding result in Theorem~\ref{thm:main-upper}.

Finally, note that the above propositions imply that
the error of the trivial algorithm $A_{1,d}(f):=f(0)$ goes to zero 
(as $d\to\infty$) for every $f\in \mathscr C^1_d(A_d,B_d,K_d)$, 
where $A_d$, $B_d$ and $K_d$ satisfy the above conditions. 

Obviously, $A_{1,d}$ uses only one function evaluation.
Therefore, 
\[
n\big(\eps, \mathscr C^1_d(A_d,B_d,K_d)\big) \,=\, 1
\]
for all $\eps\in(0,1)$ given that $d\ge d(\eps)$ is large enough (see Section~\ref{sec:IBC}). Here we also used that $n\big(\eps, \mathscr C^1_d(A_d,B_d,K_d)\big)>0$ for $\eps\in(0,1)$, 
because the initial error equals 1. In particular, we do not have the curse of dimensionality.

Finally note that, by Proposition~\ref{prop:upper1}, Proposition \ref{prop:upper2} and the fact that the isotropic constant is uniformly bounded for all $\ell_p$-balls, 
the ``only if''-part of Corollary~\ref{cor:lp} also holds for $1\le p<2$.

\bibliographystyle{abbrv}
\bibliography{curse_generalized_domains}

\end{document}